\documentclass[12pt,a4paper]{article}
\usepackage[utf8]{inputenc}
\usepackage{amsmath}
\usepackage{amsfonts}
\usepackage{indentfirst}
\usepackage{enumerate}
\usepackage{enumitem}
\usepackage{amssymb}
\usepackage{graphicx}
\usepackage{amsthm}
\usepackage{hyphenat}
\usepackage{verbatim}
\usepackage{bold-extra}
\usepackage{changepage}
\author{Luke Warren}
\title{Constructing a quasiregular analogue of $z\exp(z)$ in dimension 3}
\date{}

\begin{document}

\maketitle
\newtheorem{thm}{Theorem}[section]
\newtheorem{lem}[thm]{Lemma}
\newtheorem{cor}[thm]{Corollary}
\newtheorem{defn}[thm]{Definition}
\newtheorem{prop}[thm]{Proposition}
\newtheorem{thm*}{Theorem}
\newtheorem{lem*}{Lemma}
\newtheorem{prop*}{Proposition}
\newtheorem{pf*}{Proof of Theorem}\numberwithin{equation}{section}
\theoremstyle{definition}
\newtheorem{ex}[thm]{Example}
\newtheorem{ex*}{Example}

\newcommand{\capacity}{\operatorname{cap}}
\newcommand{\card}{\operatorname{card}}
\newcommand{\dist}{\operatorname{dist}}
\newcommand{\real}{\operatorname{Re}}
\newcommand{\im}{\operatorname{Im}}

	\begin{adjustwidth}{1cm}{1cm}
		\textbf{Abstract.} We construct a quasiregular analogue of the function $z\exp(z)$ in dimension 3, which gives the first explicit example of a quasiregular mapping of transcendental type that has exactly one zero. We then modify the construction to create a family of such quasiregular mappings and study their dynamics. From this, we also construct the first quasimeromorphic mappings with an essential singularity at infinity where the backward orbit of infinity is non-empty and finite.
	\end{adjustwidth}

\section{Introduction}

For transcendental meromorphic mappings $f: \mathbb{C} \to \hat{\mathbb{C}}$, Nevanlinna theory can be used to study the distribution of values; this includes a defect relation which yields Picard's theorem as a corollary (see \cite{Hayman64} and \cite{Nevanlinna25}). Ahlfors \cite{Ahlfors35} developed a parallel theory to that of Nevanlinna which is more geometric in nature, including a pointwise version of the same defect relation. This version states that if $f$ is a non-constant transcendental meromorphic mapping, then there exists some small set $E \subset [1,\infty)$ such that for any distinct points $a_{1}, a_{2}, \dots, a_{n} \in \hat{\mathbb{C}}$, then
\begin{equation}\label{DefectRelation}
\limsup_{\substack{r \not \in E, \\ r \to \infty}}\sum_{i=1}^{n} \delta(a_{i},r) \leq 2.
\end{equation}
Here if $a \in \hat{\mathbb{C}}$ and $r>0$, then $\delta(a,r) \in [0,1]$, called the defect function of $a$, measures the rarity of the preimages of $a$ under $f$ against the average point of $\hat{\mathbb{C}}$ in the closed ball $\overline{B}(0,r)$. In particular if $a$ is an exceptional point of $f$, so $f^{-1}(a)$ is finite or empty, then $\delta(a,r) \to 1$ as $r \to \infty$. For example, for $z \mapsto \exp(z)/z$, then $\delta(0,r) \to 1$ and $\delta(\infty, r) \to 1$ as $r \to \infty$.

A converse result to the above was proven by Drasin in \cite{Drasin77}: for each $m \in \mathbb{N}$, let $a_{m} \in \hat{\mathbb{C}}$ be distinct points and let $\delta_{m} \in [0,1]$ with $\sum \delta_{m} \leq 2$. Then there exists some non-constant transcendental meromorphic function $g$ and some small set $E \subset [1,\infty)$ such that 
\begin{equation*}
	 \lim\limits_{r \to \infty, r \not \in E} \delta(a,r)= 
\left\lbrace \begin{array}{ll}
	\delta_{m} & \text{if } a=a_{m} \text{ for some } m \in \mathbb{N}, \\
	0 & \text{otherwise}.
\end{array}\right.
\end{equation*}
%It will be useful later to note that we can also write $z\exp(z)= \exp(G(\log(z)))$, where $G(z) = z + \exp(z)$.  

Quasiregular mappings and quasimeromorphic mappings defined on higher-dimensional Euclidean space $\mathbb{R}^d$, $d \geq 2$, generalise analytic and meromorphic functions on the plane respectively. When a quasiregular or quasimeromorphic mapping defined on $\mathbb{R}^d$ has an essential singularity at infinity, we say it is of transcendental type. We shall defer the definition of quasiregular and quasimeromorphic mappings to Section~2.
	
For a quasimeromorphic mapping of transcendental type $f: \mathbb{R}^d \to \hat{\mathbb{R}}^d$, where $d \geq 2$ and $\hat{\mathbb{R}}^d = \mathbb{R}^d \cup \{\infty\}$, Rickman \cite{Rickman92} managed to generalise Ahlfors' result to establish a pointwise defect relation. Here \eqref{DefectRelation} holds almost exactly, albeit with the value 2 replaced by a constant $q$ that depends only on the dimension $d$ and the dilatation $K$ of $f$. Furthermore in dimension $d=3$, Rickman extended Drasin's theorem to the quasimeromorphic setting, once again replacing the value 2 with the constant $q$ as before. For further details, we refer to \cite{Rickman3}.

Analogous to the meromorphic setting, a corollary of Rickman's defect relation is a quasimeromorphic version of Picard's theorem (see also \cite{Rickman80}). This result was shown to be sharp for dimension $d=3$ in \cite{Rickman85} and, more recently, in all dimensions $d \geq 3$ by Drasin and Pankka \cite{DP}. They showed that given any $y_{1}, y_{2}, \dots, y_{p} \in \mathbb{R}^d$, there exists a quasiregular map $f:\mathbb{R}^d \to \mathbb{R}^d$ omitting exactly $y_{1}, y_{2}, \dots, y_{p}$. As a very simple example, Zorich maps, which form quasiregular analogues of $z \mapsto \exp(z)$, omit both 0 and $\infty$; for their general construction, see for instance \cite{Zorich} or \cite[Section~6]{IM3}.

Although Drasin and Pankka's result can give us a quasiregular function with prescribed omitted points, none of the results above can be directly applied to get a quasimeromorphic mapping with at least one non-omitted exceptional point. Furthermore, in dimensions $d \geq 3$ there are currently no known quasiregular mappings of transcendental type where a value $x \in \mathbb{R}^d$ is taken at least once, but finitely often. 

We shall construct the first explicit example of a quasiregular mapping of transcendental type in $\mathbb{R}^3$ where a value is taken finitely often. As an immediate corollary, we will provide the first example of a quasimeromorphic mapping of transcendental type in $\mathbb{R}^3$ with a single omitted pole.

\begin{thm}\label{ThmConstruction}
There exists a quasiregular mapping of transcendental type $F:\mathbb{R}^3 \to \mathbb{R}^3$ such that 
	\begin{equation}\label{Thm1.1}
	F(x)=0 \text{ if and only if } x=0.
	\end{equation}
\end{thm}

\begin{cor}\label{CorConstruction}
	There exists a quasimeromorphic map of transcendental type $f:\mathbb{R}^3 \to \hat{\mathbb{R}}^3$ such that 
	\begin{equation}\label{Cor1.2}
	\mathcal{O}^{-}_{f}(\infty) =\{0, \infty\}.
	\end{equation}
\end{cor}

As a remark, observe that we can write the complex function $z\exp(z) = \exp(G(\log(z)))$, where $G(z) = z + \exp(z)$. One reason for describing our function $F$ in Theorem~\ref{ThmConstruction} as an analogue of $z\exp(z)$ is that it will be defined as a composition $F = \mathcal{Z} \circ g \circ \mathcal{Z}^{-1}$, where $\mathcal{Z}$ is analogous to $\exp$ and $g$ is analogous to $G$.

When studying the dynamical behaviour of transcendental meromorphic functions on $\mathbb{C}$, different techniques are used based on whether the backward orbit of infinity is finite or infinite. Recently, the Julia set for quasimeromorphic mappings of transcendental type with at least one pole has been investigated in \cite{Warren1}. There, the analysis of the Julia set also requires different techniques based on the cardinality of the backward orbit of infinity. The example constructed in Corollary~\ref{CorConstruction} therefore shows that such maps exist, justifying the necessity for the different techniques used in \cite{Warren1}.

The construction method used in the proof of Theorem~\ref{ThmConstruction} readily generalises to create other quasiregular mappings that satisfy \eqref{Thm1.1}. In particular, we exhibit a family of quasiregular mappings in dimension $d=3$ for which the quasi-Fatou set is connected and coincides with the attracting basin of 0. Using this quasiregular family, we are able to construct a particular family of quasimeromorphic mappings that satisfy \eqref{Cor1.2}, for which we can give explicit points whose iterates are neither bounded nor tend to infinity; this is summarised in the following result. These two families of quasiregular and quasimeromorphic mappings may be of independent interest.

\begin{thm}\label{Thm1.3Dynamics}
	There exists a family $\mathcal{F}$ of quasiregular mappings of transcendental type such that for all $f \in \mathcal{F}$, $f$ satisfies \eqref{Thm1.1}, $J(f)$ contains half-rays and $QF(f)$ is connected. 
	
	Moreover, there exists a family $\mathcal{G}$ of quasimeromorphic mappings of transcendental type such that for all $g \in \mathcal{G}$, $g$ satisfies \eqref{Cor1.2} and there exist half-rayson which the iterates of $g$ neither stay bounded nor tend to infinity.
\end{thm}
\section{Preliminary results}
	
	\subsection{Quasiregular and quasimeromorphic mappings}
	
For notation, for $d \geq 2$ and $x \in \mathbb{R}^d$ we denote the $d$-dimensional ball centered at $x$ of radius $r>0$ as $B(x,r) = \{y \in \mathbb{R}^d : |x-y|<r\}$. For $R \in \mathbb{R}$, we shall denote the upper half-space $\{(y_{1},y_{2},y_{3}) \in \mathbb{R}^3 : y_{3} >R\}$ by $\mathbb{H}_{>R}$. We further define the upper and lower half-spaces $\mathbb{H}_{\geq R},\mathbb{H}_{\leq R}$ and $\mathbb{H}_{<R}$ similarly.

We will briefly recall the definition and some useful results for quasiregular and quasimeromorphic mappings here; see \cite{MRV1} and \cite{Rickman3}, for example, for a more comprehensive study of these mappings.
	
Let $d \geq 2$ and $G \subset \mathbb{R}^d$ be a domain. For $1 \leq p < \infty$, the Sobolev space $W_{p,loc}^{1}(G)$ is the collection of all functions $f:G \to \mathbb{R}^d$ for which all first order weak partial derivatives exist and are locally in $L^p$. Then we say that $f \in  W_{d,loc}^{1}(G)$ is quasiregular if it is continuous and there exists some constant $K\geq 1$ such that 
\begin{equation}\label{qrDefn1}
\left(\sup_{|h|=1}|Df(x)(h)|\right)^d \leq KJ_{f}(x) \text{ a.e.},
\end{equation}
where $Df(x)$ denotes the formal derivative of $f(x)$ and $J_{f}(x)$ denotes the Jacobian determinant. We denote the smallest constant $K$ for which \eqref{qrDefn1} holds by $K_{O}(f)$ and call this the outer dilatation of $f$.

If $f$ is quasiregular, then there also exists some $K' \geq 1$ such that
\begin{equation}\label{qrDefn2}
K'\left(\inf_{|h|=1}|Df(x)(h)|\right)^d \geq J_{f}(x) \text{ a.e.}
\end{equation}
We denote the smallest constant $K'$ for which \eqref{qrDefn2} holds by $K_{I}(f)$ and call this the inner dilatation of $f$. Now, the dilatation $K(f)$ of $f$ is defined as the maximum of $K_{O}(f)$ and $K_{I}(f)$. Finally if $K(f) \leq K$ for some $K \geq 1$, we say that $f$ is $K$-quasiregular.
	
The definition of quasiregularity can be extended to mappings into $\hat{\mathbb{R}}^d := \mathbb{R}^d \cup \{\infty\}$. For a domain $D \subset \mathbb{R}^d$, a continuous map $f : D \to \hat{\mathbb{R}}^d$ is called quasimeromorphic if every $x \in D$ has a neighbourhood $U_{x}$ such that either $f$ or $M \circ f$ is quasiregular from $U_{x}$ into $\mathbb{R}^d$, where $M:\hat{\mathbb{R}}^d \to \hat{\mathbb{R}}^d$ is a sense-preserving M\"{o}bius map such that $M(\infty) \in \mathbb{R}^d$.

If $f$ and $g$ are quasiregular mappings, with $f$ defined in the range of $g$, then the composition $f \circ g$ is quasiregular and the dilatation satisfies
\begin{equation}\label{compositionDilatation}
K(f \circ g) \leq K(f)K(g).
\end{equation}
	
Similarly, if $g$ is a quasiregular mapping and $f$ is a quasimeromorphic mapping defined in the range of $g$, then $f \circ g$ is quasimeromorphic and the above inequality also holds.

Many properties of analytic and meromorphic mappings have analogues within the setting of quasiregular and quasimeromorphic mappings. For example, Reshetnyak \cite{Reshetnyak1, Reshetnyak2} showed that every non-constant $K$-quasiregular map is discrete, open and sense-preserving, whilst Rickman \cite{Rickman80} extended Picard's theorem to the quasimeromorphic setting.

\subsection{Julia set of quasiregular mappings}

For notation, given $x \in \mathbb{R}^d$ we write $\mathcal{O}^{-}_{f}(x)$ to denote the backward orbit of $x$ under $f$, while we use $\mathcal{O}^{+}_{f}(x)$ to denote the forward orbit of $x$ under $f$. 

By following Bergweiler in \cite{Bergweiler6} and Bergweiler and Nicks in \cite{BN1}, for a quasiregular map $f:\mathbb{R}^d \to \mathbb{R}^d$, where $d \geq 2$, we define the Julia set of $f$ as
\begin{align} \label{JuliaSetDefn2}
J(f):= \{ &x \in \mathbb{R}^d : \capacity(\mathbb{R}^d \setminus \mathcal{O}^{+}_{f}(U_{x}))=0 \text{ for all} \nonumber \\ 
&\text{neighbourhoods } U_{x} \subset \mathbb{R}^d \text{ of } x \},
\end{align}
and define the quasi-Fatou set of $f$ as $QF(f) := \mathbb{R}^d \setminus J(f)$. Here, $\capacity(C)$ denotes the conformal capacity of a closed set $C \subset \hat{\mathbb{R}}^d$; see \cite{Rickman3} for the definition and further details. It is worth noting that sets of capacity zero are considered `small', in the sense that countable sets have capacity zero, whilst sets of capacity zero have Hausdorff dimension zero.

It has been shown that for quasiregular mappings of transcendental type, the Julia set $J(f)$ is closed, infinite and completely invariant. Further, for many quasiregular mappings of transcendental type, $J(f)$ has many other properties analogous to the classical Julia set for analytic functions on $\mathbb{C}$, such as the property that $J(f) = J(f^k)$ for any $k \in \mathbb{N}$; see \cite{BN1}.

\section{Proof of Theorem~\ref{ThmConstruction}}

For Theorem~\ref{ThmConstruction}, we shall construct a quasiregular mapping of transcendental type $F:\mathbb{R}^3 \to \mathbb{R}^3$ such that $F^{-1}(0)=\{0\}$. This shall be done through the composition of a quasiregular mapping of transcendental type $g:\mathbb{R}^3 \to \mathbb{R}^3$ and a modified version of the Zorich mapping $Z:\mathbb{R}^3 \to \mathbb{R}^3 \setminus \{0\}$, both constructed by Nicks and Sixsmith in \cite{NS2}. It will then follow that $f := M \circ F : \mathbb{R}^3 \to \hat{\mathbb{R}}^3$ is a quasimeromorphic mapping of transcendental type satisfying \eqref{Cor1.2}, where $M$ is a sense-preserving M\"{o}bius map of $\hat{\mathbb{R}}^3$ such that $M(0) = \infty$ and $M(\infty) =0$.

For $u,v \in \mathbb{R}$, let $R_{(u,v)}(x)$ denote the point attained by rotating the point $x \in \mathbb{R}^3$ by $\pi$ radians around the line $\{(u,v,t) : t \in \mathbb{R}\}$. Using this notation, observe that for any $x \in \mathbb{R}^3$ and any $\alpha,a,b \in \mathbb{R}$, if $y = R_{(u,v)}(x)$ for some $u,v \in \mathbb{R}$ then 
\begin{equation}\label{RotationRelation}
\alpha y +(a,b,0) = \alpha R_{(u,v)}(x) + (a,b,0) = R_{(\alpha u+a,\alpha v+b)}(\alpha x+(a,b,0)).
\end{equation}

Now, let $Z:\mathbb{R}^3 \to \mathbb{R}^3\setminus \{0\}$ be the Zorich-type map defined in \cite[Section~5]{NS2}, constructed as follows.

Firstly, for $(x_1, x_2, x_3) \in [-1,1]^2 \times \mathbb{R}$, set 

\begin{equation*}
Z((x_1, x_2, x_3)) := \exp(x_3)\left( x_1, x_2, 1-\max\{|x_1|, |x_2|\}\right).
\end{equation*}

We then extend $Z$ to a quasiregular mapping on $\mathbb{R}^3$ in the usual way; for every reflection in the face of the domain, we reflect in the plane $\{(y_{1},y_{2},y_{3}) : y_3=0\}$ in the image. 

It can be shown that $Z$ is periodic in the $x_1$ and $x_2$ directions with period 4, and that $Z$ satisfies the relation
\begin{equation} \label{ZorichRotation}
Z(R_{(1,1)}(x)) = Z(x).
\end{equation}

In addition, it can be seen from the definition that if $x \in \mathbb{R}^3$ and $c \in \mathbb{R}$, then
\begin{equation}\label{ZorichTranslation}
Z(x + (0,0,c)) = \exp(c)Z(x).
\end{equation}

Now let $g: \mathbb{R}^3 \to \mathbb{R}^3$ and $L>1$ denote the quasiregular map of transcendental type and the constant from \cite[Section~6]{NS2} respectively. It was shown there that using $Z$ as defined above, this map $g$ has the following properties:

\begin{enumerate}
\item[(G1)] $g(x)=x$ when $x \in \mathbb{H}_{\leq 0}$,
\item[(G2)] $g(x)=x+Z(x)$ when $x \in \mathbb{H}_{\geq L}$,
\item[(G3)] $g(x+c) = g(x)+c$ for $c \in \{(4,0,0), (0,4,0)\}$, and
\item[(G4)] $g(R_{(2,2)}(x)) = R_{(2,2)}(g(x))$.
\end{enumerate}

\subsection{Construction of the quasiregular map $F$}

By using the properties of $g$ and $Z$ above, we shall now proceed to the construction of $F$. First, define the translation $T:\mathbb{R}^3 \to \mathbb{R}^3$ by $T(x)=x-(1,1,0)$. Now define $F: \mathbb{R}^3 \to \mathbb{R}^3$ by setting $F(0)=0$, and for every $x \in \mathbb{R}^3 \setminus \{0\}$ set
\begin{equation*}
F(x) = (Z \circ T \circ g \circ T^{-1} \circ \phi)(x),
\end{equation*}
where $\phi$ is an inverse branch of $Z$. We claim that $F$ is independent of the choice of $\phi$ and that $F$ is a quasiregular mapping of transcendental type satisfying \eqref{Thm1.1}.
Initially, we will assume that $F$ is independent of the choice of $\phi$. Note that $F^{-1}(0)=\{0\}$ and
\begin{equation}\label{Schroder1}
(F \circ (Z \circ T))(x) = ((Z \circ T) \circ g)(x) \text{ for all } x \in \mathbb{R}^3 \setminus \{0\}.
\end{equation}

By property (G1) of $g$, we find that $(T^{-1} \circ \phi)(w)$ is a fixed point of $g$ for any $w \in B(0, 1/2) \setminus \{0\}$. It follows by construction that $F$ fixes all points in $B(0,1/2) \setminus \{0\}$, hence $F$ is continuous at 0. Further, for each $x \in \mathbb{R}^3 \setminus \{0\}$ we can choose an inverse branch $\phi$ such that $\phi$ is continuous at $x$. Now as $F$ is the composition of continuous mappings, we have that $F$ is also continuous on $\mathbb{R}^3 \setminus \{0\}$.

Next consider the branch set of $Z$, given by $B_{Z}= \{(2n+1, 2m+1, x_{3}) : n,m \in \mathbb{Z}, x_{3} \in \mathbb{R}\}$. By a direct calculation,
\begin{equation*}
Z(B_{Z}) \cup \{0\} = \{t(1,1,0) : t \in \mathbb{R}\}  \cup \{t(-1,1,0) : t \in \mathbb{R}\}.
\end{equation*} 

Let $V=Z(B_{Z}) \cup \{0\}$. Then there exists some $K \geq 1$ depending only on $Z$ such that for any $x \in \mathbb{R}^3 \setminus V$, it is possible to choose an inverse branch $\phi$ of $Z$ that is locally $K$-quasiconformal on some neighbourhood $U_{x} \subset \mathbb{R}^d \setminus V$ of $x$. As the functions $Z$,$g$,$T$ and $T^{-1}$ are quasiregular and the inverse branch $\phi$ can be chosen to be locally $K$-quasiconformal, then $F$ will be quasiregular on $\mathbb{R}^3 \setminus V$. Since $V$ has zero measure and $F$ is continuous on $\mathbb{R}^3$, it follows that $F$ is a quasiregular mapping of transcendental type on $\mathbb{R}^3$.

It remains to prove the claim that $F$ is independent of the choice of the inverse branch of $Z$. To this end, let $x \in \mathbb{R}^3 \setminus \{0\}$ and suppose that $\psi \neq \phi$ is a different inverse branch of $Z$ defined at $x$. Set $F_{1} := Z \circ T \circ g \circ T^{-1} \circ \psi$. 

By the construction of $Z$, there exist some $n,m \in \mathbb{Z}$ and $p \in \{0,1\}$ such that 
\begin{equation}\label{PointRelation}
\phi(x) = R^{p}_{(1,1)}(\psi(x)) + (4n,4m,0),
\end{equation}
where $R^{1}_{(1,1)} = R_{(1,1)}$ and $R^{0}_{(1,1)}$ denotes the identity mapping.

Let $u:= (T^{-1} \circ \phi)(x)$ and let $v:=(T^{-1} \circ \psi)(x)$. From \eqref{RotationRelation} and \eqref{PointRelation} we now get
\begin{equation*}
u = R^{p}_{(1,1)}(\psi(x)) + (1,1,0) + (4n,4m,0) = R^{p}_{(2,2)}(v) + (4n,4m,0).
\end{equation*}

It then follows from properties (G3) and (G4) of $g$ that
\begin{equation}\label{gPointRelation}
g(u) = g(R^{p}_{(2,2)}(v) + (4n,4m,0)) = R^{p}_{(2,2)}(g(v))+(4n,4m,0).
\end{equation}

By appealing to \eqref{RotationRelation}, then \eqref{gPointRelation} yields
\begin{equation*}
T(g(u)) = g(u)-(1,1,0) = R^{p}_{(1,1)}(g(v)-(1,1,0))+(4n,4m,0).
\end{equation*}

Finally, it follows from \eqref{ZorichRotation} and the periodicity of $Z$ that
\begin{align*}
F(x) &= (Z \circ T)(g(u)) \\
&= Z(R^{p}_{(1,1)}(g(v)-(1,1,0))+(4n,4m,0)) \\
&= Z(g(v)-(1,1,0)) \\
&= (Z \circ T)(g(v)) = F_{1}(x),
\end{align*}
therefore the claim follows.

As a remark, since $F$ satisfies \eqref{Schroder1} then $F$ and $g$ are semi-conjugate to each other by the quasiregular mapping $Z \circ T$. Thus for each $n \in \mathbb{N}$,
\begin{equation}\label{MapIteration}
F^n = Z \circ T \circ g^n \circ T^{-1} \circ \phi.
\end{equation}

\begin{proof}[Proof of Corollary~\ref{CorConstruction}]
Let $F$ be as above and let $M:\mathbb{R}^3 \to \hat{\mathbb{R}}^3$ be the sense-preserving M\"{o}bius map defined by 
\begin{equation}\label{Mobius}
M((x_{1},x_{2},x_{3})) = \frac{1}{x_{1}^2 + x_{2}^2 + x_{3}^2}(x_{1},x_{2},-x_{3}).
\end{equation}

As $M$ is quasimeromorphic and $F$ is quasiregular of transcendental type, then $f:= M \circ F:\mathbb{R}^3 \to \hat{\mathbb{R}}^3$ is a quasimeromorphic mapping of transcendental type. Further by the definition of $M$ and the fact that $F^{-1}(0)=\{0\}$, then $\mathcal{O}^{-}_{f}(\infty) = \{0,\infty\}$ as required.
\end{proof}

\section{Modifying the construction of $F$}
To create other quasiregular mappings of transcendental type with a value taken finitely often, we can consider replacing the function $g$ with a quasiregular function of transcendental type $\tilde{g}:\mathbb{R}^3 \to \mathbb{R}^3$ that satisfies the following properties:

\begin{enumerate}
\item[(I)] for $c \in \{(4,0,0), (0,4,0)\}$, then $\tilde{g}(x+c)=\tilde{g}(x) + \alpha c$ for some $\alpha \in \mathbb{Z}$,
\item[(II)] $\tilde{g}(R_{(2,2)}(x)) = R_{(2,2)}(\tilde{g}(x))$, and
\item[(III)] for every $M \geq 0$, there exists some $N \geq 0$ such that $\tilde{g}(x) \in \mathbb{H}_{\leq -M}$ whenever $x \in \mathbb{H}_{\leq -N}$.
\end{enumerate}

Let $\tilde{F}: \mathbb{R}^3 \to \mathbb{R}^3$ be defined by $\tilde{F}(0)=0$ and $\tilde{F}(x)= (Z \circ T \circ \tilde{g} \circ T^{-1} \circ \phi)(x)$ for all $x \in \mathbb{R}^3 \setminus \{0\}$, where $\phi$ is an inverse branch of $Z$. 

To show that $\tilde{F}$ is continuous at 0, let $\varepsilon >0$ be given and observe that $x \in \mathbb{H}_{\leq \log(\varepsilon)}$ implies that $(Z \circ T)(x) \in B(0, 2\varepsilon)$. By (III), there exists some $N \geq 0$ such that $\tilde{g}(y) \in \mathbb{H}_{\leq \log(\varepsilon)}$ whenever $y \in \mathbb{H}_{\leq -N}$. Therefore by taking $\delta = \exp(-N)/2>0$, then $w \in B(0,\delta)$ implies that $(T^{-1} \circ \phi)(w) \in \mathbb{H}_{\leq -N}$ and the continuity of $\tilde{F}$ follows.

Applying similar arguments to those in Section~3.1, we now get that $\tilde{F}$ is a well-defined quasiregular mapping of transcendental type that also satisfies
\begin{equation}\label{Schroder2}
(\tilde{F} \circ (Z \circ T))(x) = ((Z \circ T) \circ \tilde{g})(x) \text{ for all } x \in \mathbb{R}^3 \setminus \{0\}.
\end{equation}
From here, composing $\tilde{F}$ with $M$ from \eqref{Mobius} gives us another quasimeromorphic mapping of transcendental type that satisfies \eqref{Cor1.2}.

As before, $\tilde{F}$ and $\tilde{g}$ are semi-conjugate to each other by $Z \circ T$, and the iterates of $\tilde{F}$ and $\tilde{g}$ are related as in \eqref{MapIteration}. Consequently, there is a connection between the dynamics of $\tilde{F}$ and $\tilde{g}$. 

\begin{lem}\label{ModifyGJuliaSet}
Let $\tilde{g}$ be a quasiregular mapping satisfying \emph{(I)} - \emph{(III)} and let $\tilde{F} :=  Z \circ T \circ \tilde{g} \circ T^{-1} \circ \phi$. Then $(Z \circ T)(J(\tilde{g})) \subset J(\tilde{F})$. 
\end{lem}

\begin{proof}
Let $x \in (Z \circ T)(J(\tilde{g}))$ and let $U_{x} \subset \mathbb{R}^3$ be an arbitrary open neighbourhood of $x$. Since $Z \circ T$ is continuous and open, then there exists some $y \in J(\tilde{g})$ and an open neighbourhood $U_{y} \subset \mathbb{R}^3$ of $y$ such that $(Z \circ T)(y) = x$ and $(Z \circ T)(U_{y}) \subset U_{x}$. As $\tilde{F}$ is independent of the choice of inverse branch of $Z$, then choose a branch $\phi$ such that $(T^{-1} \circ \phi)(U_{x}) \supset U_{y}$.

As $y \in J(\tilde{g})$, then by \eqref{JuliaSetDefn2} there exists some set $X \subset \mathbb{R}^3$ with $\capacity(X)=0$ such that
\begin{equation*}
\bigcup_{k=0}^{\infty}\tilde{g}^k(U_{y}) \supset \mathbb{R}^3 \setminus X.
\end{equation*}
It follows from \cite[Theorem~7.1]{MRV1} that $Y:= (Z \circ T)(X) \cup \{0\}$ is such that $\capacity(Y)=0$. As $Z(\mathbb{R}^3) = \mathbb{R}^3 \setminus \{0\}$, then 
\begin{equation*}
(Z \circ T)\left(\bigcup_{k=0}^{\infty}\tilde{g}^k(U_{y})\right) \supset (Z \circ T)(\mathbb{R}^3) \setminus (Z \circ T)(X) = \mathbb{R}^3 \setminus Y.
\end{equation*}

Now observe that
\begin{align*}
\bigcup_{k=0}^{\infty}\tilde{F}^k(U_{x}) &= \bigcup_{k=0}^{\infty}(Z \circ T)(\tilde{g}^k((T^{-1} \circ \phi)(U_{x})))\\
&\supset \bigcup_{k=0}^{\infty}(Z \circ T)(\tilde{g}^k(U_{y}))\\
&= (Z \circ T)\left(\bigcup_{k=0}^{\infty}\tilde{g}^k(U_{y})\right) \supset \mathbb{R}^3 \setminus Y.
\end{align*}

As $U_{x}$ was an arbitrary neighbourhood of $x$, then we have that $x \in J(\tilde{F})$ and hence $(Z \circ T)(J(\tilde{g})) \subset J(\tilde{F})$.
\end{proof}

\section{Dynamics of a family of quasiregular maps}
Let $F$ and $g$ be the quasiregular mappings from Section~3. Then an example of an infinite family of quasiregular mappings satisfying \eqref{Thm1.1} is $\mathcal{F}_{\lambda_{0}} :=\{\lambda F: 0<\lambda<\lambda_{0}\}$. These mappings can be constructed by considering the quasiregular maps $g_{t}: \mathbb{R}^3 \to \mathbb{R}^3$ defined by $g_{t}(x) = g(x) + (0,0,t)$ for each $t \in \mathbb{R}$, which all satisfy (I) - (III). It follows from \eqref{ZorichTranslation} that for all $\lambda >0$ and any inverse branch $\phi$ of $Z$,
\begin{equation}\label{LambdaFDefinition}
\lambda F = Z \circ T \circ g_{\log{\lambda}} \circ T^{-1} \circ \phi.
\end{equation}
The aim of this section is to establish some dynamical results for the family $\mathcal{F}_{\lambda_{0}}$ when $\lambda_{0}>0$ is sufficiently small, culminating in Proposition~\ref{JfandQFfRelations} below. The dynamics of some quasimeromorphic mappings associated with $\mathcal{F}_{1}$ can also be considered, yielding Lemma~\ref{QmDynamics} below; this together with Proposition~\ref{JfandQFfRelations} will prove Theorem~\ref{Thm1.3Dynamics}.

It was shown in \cite[Section~7]{NS2} that there exists a constant $0<C<1$ such that if $0< \lambda \leq C$, then
\begin{equation*}
g_{\log{\lambda}}(\mathbb{H}_{\leq L}) \subset \mathbb{H}_{< 0},
\end{equation*}
where $L >1$ is the constant from property (G2) of $g$. Moreover for such values of $\lambda>0$, it was shown that the quasi-Fatou set $QF(g_{\log{\lambda}})$ consists of a single connected domain containing the lower half-space $\mathbb{H}_{< 0}$, in which all iterates of $g_{\log{\lambda}}$ tend to infinity locally uniformly. In particular, for every $x \in QF(g_{\log{\lambda}})$ there exists some $k \in\mathbb{N}$ such that $g_{\log{\lambda}}^k(x) \in \mathbb{H}_{< 0}$.

By using the semi-conjugacy of \eqref{LambdaFDefinition}, we shall first give a result relating the quasi-Fatou set of $g_{\log\lambda}$ to the attracting basin of $\lambda F$ at 0. Here, for a general mapping $f$, the attracting basin of $0$ with respect to $f$ is defined by $\mathcal{A}_{f}(0):= \{x \in \mathbb{R}^3 : f^n(x) \to 0 \text{ as } n \to \infty\}$.

\begin{lem}\label{AttractingBasin}
Let $0<\lambda \leq C$. Then $(Z \circ T)(QF(g_{\log{\lambda}})) \subset \mathcal{A}_{\lambda F}(0)$.
\end{lem}

\begin{proof} 

Let $x \in (Z \circ T)(QF(g_{\log{\lambda}}))$, so there exists some $y \in QF(g_{\log{\lambda}})$ such that $(Z \circ T)(y)=x$. As $\lambda \leq C<1$, there exists some $k \in \mathbb{N}$ such that $g_{\log{\lambda}}^{k}(y) \in \mathbb{H}_{< 0}$. Now by choosing $\phi$ such that $\phi(x)=T(y)$, then using \eqref{LambdaFDefinition} and (G1) we get that for all $n \in \mathbb{N}$,
\begin{equation*}
(\lambda F)^{n+k}(x) = (Z \circ T \circ g_{\log{\lambda}}^n \circ g_{\log{\lambda}}^k \circ T^{-1} \circ \phi)(x)
= (Z \circ T \circ g_{\log{\lambda}^n})(g_{\log{\lambda}}^k(y)).
\end{equation*}

It follows from \eqref{ZorichTranslation} that
\begin{equation*}
(Z \circ T \circ g_{\log{\lambda}^n})(g_{\log{\lambda}}^k(y)) =\lambda^{n}(Z \circ T)(g_{\log{\lambda}}^k(y)) \to 0 \text{ as } n \to \infty,
\end{equation*}
and the proof follows.
\end{proof}

By using the above lemma, we can establish a strong relationship between the Julia sets and quasi-Fatou sets of $g_{\log{\lambda}}$ and $\lambda F$ as follows.

\begin{prop}\label{JfandQFfRelations}
Let $0<\lambda\leq C$ and let $L>1$ be the constant from property \emph{(G2)} of $g$. Then 
\begin{itemize}
	\item[\emph{(i)}] $J(\lambda F) = (Z \circ T)(J(g_{\log{\lambda}}))$ contains the half-rays $\{t(1,1,0) : |t|>\max\{\log(1/\lambda), \exp(L)\}\}$, and
	\item[\emph{(ii)}] $QF(\lambda F) = \mathcal{A}_{\lambda F}(0) = (Z \circ T)(QF(g_{\log{\lambda}})) \cup \{0\}$ is connected and contains $B(0,1/2)$.
\end{itemize}
\end{prop}

\begin{proof}
	
Firstly, recall that $F$ is the identity on $B(0,1/2)$. It immediately follows that for all $0<\lambda<1$, then $B(0,1/2) \subset \mathcal{A}_{\lambda F}(0) \cap QF(\lambda F)$. As $\mathcal{A}_{\lambda F}(0)$ and $QF(\lambda F)$ are completely invariant, then
\begin{equation}\label{AIsInQF}
\mathcal{A}_{\lambda F}(0) \subset QF(\lambda F).
\end{equation} 
		
To prove that $J(\lambda F) = (Z \circ T)(J(g_{\log{\lambda}}))$, observe that from Lemma~\ref{ModifyGJuliaSet} we have $(Z \circ T)(J(g_{\log{\lambda}})) \subset J(\lambda F)$. For the reverse inclusion, let $x \in J(\lambda F)$. As $x \neq 0$, then there exists some $y \in \mathbb{R}^3$ that $(Z\circ T)(y)=x$. It then follows from Lemma~\ref{AttractingBasin} and \eqref{AIsInQF} that $y \not \in QF(g_{\log \lambda})$. Therefore $y \in J(g_{\log \lambda})$, as required.

Now by a remark in \cite[Section~7]{NS2}, it was shown that
\begin{equation*}
\{(4n +c,4m +c,x_{3}) : n,m \in \mathbb{Z}, c \in \{0,2\}, x_{3}>\log(L')\} \subset J(g_{\log{\lambda}}),
\end{equation*}
where direct calculation yields $L' = \max\{\log(1/\lambda), \exp(L)\} >1$. The first part of (i) now implies $\{t(1,1,0) : |t|>L'\} \subset J(\lambda F)$, completing the proof of (i).

To prove (ii), note that from (i) we have $QF(\lambda F) \subset (Z \circ T)(QF(g_{\log{\lambda}})) \cup \{0\}$, since $(Z \circ T)(\mathbb{R}^3) = \mathbb{R}^3 \setminus \{0\}$. Further from Lemma~\ref{AttractingBasin} and \eqref{AIsInQF},
\begin{equation*}
QF(\lambda F) \subset (Z \circ T)(QF(g_{\log{\lambda}})) \cup \{0\}\subset \mathcal{A}_{\lambda F}(0) \subset QF(\lambda F),
\end{equation*}
thus equality is attained. Further, since $QF(g_{\log{\lambda}})$ is a single connected domain and $B(0, 1/2) \subset QF(\lambda F)$, then $QF(\lambda F)$ is connected, completing the proof.
\end{proof}

%It was remarked by Nicks and Sixsmith \cite{NS2} that there exists some large $L'>0$, depending only on $0<\lambda<1$ and the constant $L>1$ from property (G2) of $g$, such that
%\begin{equation}\label{L'Constant}
%\{(4n +c,4m +c,x_{3}) : n,m \in \mathbb{Z}, c \in \{0,2\}, x_{3}>\log{L'}\} \subset J(g_{\log{\lambda}}).
%\end{equation}

%It immediately follows from the above lemma that 
%\begin{equation*}
%\{t(1,1,0) : |t|>L'\} \subset J(\lambda F).
%\end{equation*}

Define $f_{\lambda} := M \circ \lambda F$, where $M$ is the M\"{o}bius map from \eqref{Mobius}. By considering the behaviour of $\lambda F$ and $M$, we can explicitly locate points whose sequence of iterates under $f_{\lambda}$ contains both a subsequence that always remains bounded and a subsequence that tends to infinity. We denote the set of such points by $BU(f_{\lambda})$. In particular, we can explicitly show that the half-rays from Proposition~\ref{JfandQFfRelations}(i) are in this set.

\begin{lem}\label{QmDynamics}
	Let $0<\lambda<1$, let $L>1$ be the constant from property \emph{(G2)} of $g$ and let $f_{\lambda} = M \circ \lambda F$. Then 
	\begin{equation*}
	\{t(1,1,0) : |t|> L' \} \cup \{t(1,1,0) : 0<|t|< 1/(2L')\} \subset BU(f_{\lambda}),
	\end{equation*}
	where $L' = \max\{\log(1/\lambda), \exp(L)\}$.
\end{lem} 

\begin{proof}
	Let $0<\lambda<1$ and let $L>1$ be the constant from property (G2) of $g$. Now consider the behaviour of $\lambda F$ on the line $\{t(1,1,0) : t \in \mathbb{R}\}$. Indeed by the properties of $Z$, $g_{\log{\lambda}}$ and \eqref{LambdaFDefinition}, a direct calculation yields
	\begin{equation}\label{LambdaFLineMap}
	(\lambda F)(\alpha(1,1,0)) = 
	\left\lbrace \begin{array}{ll}
	\lambda\exp(|\alpha|)\alpha(1,1,0) & \text{if } |\alpha| >\exp(L), \\
	\lambda\alpha(1,1,0) & \text{if } |\alpha|<1.
	\end{array}\right.
	\end{equation}

	It follows that as $n \to \infty$, then $(\lambda F)^n(\alpha(1,1,0)) \to 0$ when $|\alpha| <1$, while $(\lambda F)^n(\alpha(1,1,0)) \to \infty$ when $|\alpha| >\max\{\log(1/\lambda), \exp(L)\}$. 
	
	Next note that from \eqref{Mobius}, for any $\alpha \in \mathbb{R}$ we have
	\begin{equation*}
	M(\alpha(1,1,0)) = \frac{1}{2\alpha}(1,1,0).
	\end{equation*}
	
	Combining this with \eqref{LambdaFLineMap} and considering the even and odd iterates of $f_{\lambda} = M \circ \lambda F$, the result follows.
\end{proof}

\subsection*{Acknowledgements}
The author would like to thank Dave Sixsmith for many helpful comments and suggestions during the preparation of this paper.

\end{document}